\documentclass[10.5pt]{article}
\usepackage{amsmath}
\usepackage{amsfonts}
\usepackage{graphicx}

\usepackage{amsmath}
\usepackage{amssymb}
\usepackage{latexsym}
\usepackage{amsmath,amsfonts,amssymb,amsthm,euscript,makeidx,color,mathrsfs,bm}

\newtheorem{problem}{Problem}
\newtheorem{theorem}{Theorem}
\newtheorem{lemma}{Lemma}
\newtheorem{remark}{Remark}

\def\rf{\eqref}
\def\no{\nonumber}
\makeatletter
   
   \@addtoreset{equation}{section}
\makeatother

\title{
Closed-Loop Stackelberg Strategy for Linear-Quadratic Leader-Follower Game
\thanks{
This work was supported by the Original Exploratory Program Project of National Natural Science Foundation of China (62250056), the Joint Funds of the National Natural Science Foundation of China (U23A20325), the Major Basic Research of Natural Science Foundation of Shandong Province (ZR2021ZD14), and the High-level Talent Team Project of Qingdao West Coast New Area (RCTD-JC-2019-05), the National Natural Science Foundation of China (62103241,62273213), Shandong Provincial Natural Science Foundation (ZR2021QF107).
Corresponding author: Huanshui Zhang. Email: hszhang@sdu.edu.cn}
}

\author{Hongdan Li, Juanjuan Xu and Hunashui Zhang }

\begin{document}

\pagenumbering{arabic}
 \setcounter{page}{1}

\pagenumbering{arabic} \thispagestyle{empty} \setcounter{page}{1}

\baselineskip 16pt
\date{}
  \maketitle
\begin{abstract}
This paper is concerned with the closed-loop Stackelberg strategy for linear-quadratic leader-follower game. Completely different from the open-loop and feedback Stackelberg strategy, the solvability of the closed-loop solution even the linear case remains challenging. The main contribution of the paper is to derive the explicitly linear closed-loop Stackelberg strategy with one-step memory in terms of Riccati equations. The key technique is to apply the constrained maximum principle to the leader-follower game and explicitly solve the corresponding forward and backward difference equations. Numerical examples verify the effectiveness of the results, which achieves better performance than feedback strategy.
\end{abstract}

\section{Introduction}

Stackelberg solution is of great significance in plenty of fields including economics, engineering and biology \cite{Application1, Application2, Application3, Application4}. Specific examples are referred to government regulation of macro economy \cite{Government} and transboundary pollution \cite{Pollution}. The salient feature of such problems is that at least two players are involved in nonzero-sum game, well-known as Stackelberg game or leader-follower game, and one player named by the leader has the ability to enforce his strategy on the other player named by the follower. The asymmetry of players' roles makes the derivation of Stackelberg solution extremely complex, which is completely different from the symmetric positions among players in Nash game \cite{Nash}.

The dynamic Stackelberg solution was initially studied in the 1970s \cite{Dynamic} and has achieved many fruitful results since then. For example, an implicit necessary and sufficient condition was given in \cite{Openloop1} for the unique existence of the solution based on the operator technique.
\cite{Openloop2} proposed a sufficient existence condition by adopting a Lyapunov-type approach. In particular, the explicit Stackelberg solution is given for the linear-quadratic game in terms of three coupled and nonsymmetric Riccati equations. To avoid the solvability of the above nonstandard Riccati equations, \cite{Openloop3} introduced three decoupled and symmetric Riccati equations to design the Stackelberg solution, which was later extended to solve the Stackelberg game with time delay \cite{Openloop4}. The stochastic linear-quadratic leader-follower game with multiplicative noise was solved in  \cite{Openloop5}.
It is noted that the Stackelberg solutions in the above mentioned work are open-loop, that is, the information sets available to the players contain only initial value. This makes it easy to obtain the equilibrium conditions based on the maximum principle. However, it is not very reasonable to assume that the players have access to only open-loop information in a dynamic game as has been pointed in \cite{Basar79}.

Another widely studied strategy is feedback strategy, i.e., in the feedback form of the state at current time. Once the feedback structure of the Stackelberg strategy is known, the optimal feedback gain can be obtained by applying the matrix maximum principle or dynamic programming. For instance, \cite{Feedback1} derived feedback Stackelberg strategies for linear-quadratic two-player game by using dynamic programming. A multilevel feedback Stackelberg strategy was then studied in \cite{Feedback2} for
systems with $M$ players and sufficient conditions were obtained for existence and uniqueness of solution in a linear-quadratic discrete-time game.
Necessary conditions for feedback Stackelberg strategies were given in \cite{Feedback3} for linear stochastic systems governed by It\^{o} differential equations with multiple followers by
using sets of cross-coupled algebraic nonlinear matrix equations, and two numerical algorithms
were presented to solve the matrix equations based on Newton's method and the semidefinite programming problem.
\cite{Feedback4} studied the mean field Stackelberg games
between the leader and a large population of followers. By treating the mean field as part of the dynamics of the
major player and a representative minor player, the decision problems were Markovianized and the equilibrium strategy was shown in a state feedback form.

Compared with the open-loop and feedback strategies, the closed-loop strategy, whose information sets available to the players contain all the states from initial time to the current time, seems to be more advantageous because the leader's motivation for using the Stackelberg strategy is to reduce losses \cite{Closedloop}. However, its derivation faces major challenge even in the linear quadratic game where the standard techniques such as dynamic programming are invalid. The main reason lies in that the follower's reaction can only be expressed implicitly for every announced strategy of the leader which further leads to a nonstandard optimization problem for the leader containing implicit constraints \cite{BasarOnestep}.
To overcome this difficulty, one way is to study the closed-loop memoryless Stackelberg strategy, that is, the information sets available to the players are restricted be consisting of only two parts: the state at current time and the initial state. In this case, the leader faces a nonclassical optimization problem where the controlled system contains the derivative of the control with respect to the state.
\cite{CLM1} derived the necessary conditions satisfied by the leader's closed-loop memoryless
optimal strategy by applying the variational techniques to the state system
with mixed-boundary conditions and by reducing the nonclassical control problem into a classical one, respectively.
\cite{CLM2} studied the stochastic Stackelberg game and derived the maximum principle for the leader's global Stackelberg solution under the adapted closed-loop memoryless information structure which is closely related to the theory of controlled forward and backward stochastic differential equations. The other way is proposed in \cite{Basar79} which relates the closed-loop Stackelberg strategy to a particular representation of the optimal solution of a team problem. \cite{Teamoptimal1} derived the  team-optimal closed-loop Stackelberg strategies for a class of continuous-time two-player nonzero-sum differential games. \cite{Teamoptimal2} studied the team-optimal closed-loop Stackelberg strategies for systems with slow and fast modest. \cite{Teamoptimal3} investigated the team-optimal closed-loop Stackelberg strategies for discrete-time descriptor systems.

Although many important progresses have been made in the research of closed-loop strategy, the problem has not yet been solved. In this paper, we will study the closed-loop Stackelberg strategy for linear-quadratic leader-follower game. In particular, the strategy is in the linear closed-loop form with one-step memory. By applying the constrained maximum principle, the solvability of the leader-follower game is reduced to that of forward and backward difference equations. The main contribution is to derive the explicitly linear closed-loop Stackelberg strategy with one-step memory in terms of Riccat equations. Numerical examples show that the derived closed-loop strategy achieves better performance than feedback strategy.

The remainder of the paper is organized as follows. Section 2 presents the studied problem.
The main result is given in Section 3. Numerical example is given in Section 4. Some concluding remarks are given in the last section.

The following notations will be used throughout this paper: $\mathbb{R}^n$
denotes the set of $n$-dimensional vectors; $x'$ denotes the
transpose of $x$; a symmetric matrix $M>0\ (\geq 0)$ means that $M$ is
strictly positive-definite (positive semi-definite).

\section{Problem  Statement  }

\normalsize
Consider the following  linear system
\begin{eqnarray}
x_{k+1}=Ax_{k}+B_{1}u^{f}_{k}+B_{2}u^{l}_{k},\label{f2.1}
\end{eqnarray}
where $x_{k}\in \mathbb{R}^{n}$ is a state. One of the players $u^{f}_{k}\in \mathbb{R}^{m_{1}}$ is as a follower, the other $u^{l}_{k}\in \mathbb{R}^{m_{2}}$ is a leader. $A, B_{1}, B_{2}$ are real matrices with compatible dimensions.
Define the following two cost functionals
\begin{eqnarray}
 J^{f}&=&\sum^{N}_{k=0}[x'_{k}Qx_{k}+(u^{f}_{k})'R_{1}u^{f}_{k}
 +(u^{l}_{k})'R_{2}u^{l}_{k}]+x'_{N+1}P_{N+1}x_{N+1},\label{f2.3}\\
 J^{l}&=&\sum^{N}_{k=0}[x'_{k}\bar{Q}x_{k}+(u^{f}_{k})'\bar{R}_{1}u^{f}_{k}
 +(u^{l}_{k})'\bar{R}_{2}u^{l}_{k}]+x'_{N+1}\bar{P}_{N+1}x_{N+1},\label{f2.4}
\end{eqnarray}
where the state weighting matrices $Q, \bar{Q}, P_{N+1}$ and $\bar{P}_{N+1}$ are semi-positive definite. The control weighting matrices $R_{i}$ and $\bar{R}_{i}, i=1, 2$ are positive definite.

For illustrating, recall some definitions (refer to Definition 5.2 in \cite{Basar99}) that a player's information structure $\eta_{k}$ is

1) an open-loop pattern if  $\eta_{k}=\{x_{0}\}$;

2) a feedback information pattern if $\eta_{k}=\{x_{k}\}$;

3) a closed-loop information with memoryless if $\eta_{k}=\{x_{0}, x_{k}\}$;

4) a closed-loop information with memory if $\eta_{k}=\{x_{0},\ldots,x_{k}\}$.

Note that it still remains challenging to find the closed-loop Stackelberg strategies even for an LQ case since the classical optimal control methods fails. So far there is no direct method to solve the closed-loop Stackelberg strategies. However, from the definition of closed-loop  information structure, the linear closed-loop Stackelberg strategies with one-step memory  are assumed to be as follows for $0\leq k \leq N$.
\begin{eqnarray}
u^{f}_{k}&=&K_{k}x_{k}+G_{k}x_{k-1},\label{f2.5}\\
u^{l}_{k}&=&\bar{K}_{k}x_{k}+\bar{G}_{k}x_{k-1},\label{f2.6}
\end{eqnarray}
where parameter matrices $K_{k}, G_{k}\in \mathbb{R}^{m_{1}\times n}, \bar{K}_{k}, \bar{G}_{k}\in \mathbb{R}^{m_{2}\times n}$ with $G_{0}=\bar{G}_{0}=0$ are to be determined.
The problem to be discussed in this paper is as follows.
\begin{problem}\label{p1}
Find the leader's controller $u^{l}$ as defined in \rf{f2.6} to minimize \rf{f2.4} under
the condition that the follower constructs a controller $u^{f}$ as defined in \rf{f2.5} minimizing \rf{f2.3} given $u^{l}$.
\end{problem}
\begin{remark}\label{r0}
 This problem is called a linear closed-loop Stackelberg game control. It is different from the well-known state  feedback stategy where the controller is assumed to be of current state feedback, i.e., $u_{k}=K_{k}x_{k}$. It will be shown  that the controller with linear closed-loop form as \rf{f2.5}-\rf{f2.6} will achieve better performance than state feedback controller.
\end{remark}
\begin{remark}\label{r01}
It is noted that a team-optimal approach has been proposed in \cite{Basar79} to derive the closed-loop Stackelberg strategy where the explicit solution depends on both recursive equations and one kind of matrix equation. As comparison, the gain matrices in (\ref{f2.5}) and (\ref{f2.6}) will be derived in terms of Riccati equations as shown in Theorem \ref{t3} of Section 3.3 in this paper.
\end{remark}
\section{Main Result}
To solve Problem \ref{p1}, we first review the following standard LQR with a constraint of current state-feedback. Consider the following linear system
\begin{eqnarray}
x_{k+1}=Ax_{k}+Bu_{k},\label{1281}
\end{eqnarray}
and define the cost functional as follows.
\begin{eqnarray}
J_{N}=\Big\{\sum^{N}_{k=0}[x'_{k}Qx_{k}+u_{k}'Ru_{k}]+x'_{N+1}P_{N+1}x_{N+1}\Big\},\label{0}
\end{eqnarray}
where $Q,R$ are semi-positive definite. Different from the standard LQR problem, we assume that the controller is with the following form
\begin{eqnarray}
u_{k}=K_{k}x_{k}.\label{1282}
\end{eqnarray}
The constraint LQR  problem is to find the controller with form of \rf{1282} to minimize \rf{0}. Now, we have the following result:
\begin{lemma}\label{l1}
Suppose the above constrained LQR problem admits a solution, then, the controller satisfies
\begin{eqnarray}
0=Ru_{k}+B'\lambda_{k}-\beta_{k}, \label{2}
\end{eqnarray}
where $\lambda_k$ is as
\begin{eqnarray}
\left\{
  \begin{array}{ll}
\lambda_{k-1}=Qx_{k}+A'\lambda_{k}+K_{k}'\beta_{k},\ \ 0<k\leq N,\\
\lambda_{N}=P_{N+1}x_{N+1},
  \end{array}
\right. \label{02}
\end{eqnarray}
and $\beta_{k}$ is arbitrary to be determined.
\end{lemma}
\begin{proof}
For the above optimal control problem with equality constraint, introducing Lagrange multipliers $\lambda_{k}$ and $\beta_{k}$, accordingly, define the following Hamiltonian functional
\begin{eqnarray}
H_{k}=x'_{k}Qx_{k}+u_{k}'Ru_{k}+\lambda_{k}'(Ax_{k}+Bu_{k})
+\beta_{k}'(K_{k}x_{k}-u_{k}).\label{1251}
\end{eqnarray}
Now, the constrained LQR problem is equivalently transformed unconstrained optimization problem. By taking partial derivative of $u_{k}$ and $x_{k}$ for \rf{1251}, we obtain
\begin{eqnarray}
0&=&\frac{\partial H_{k}}{\partial u_{k}}=Ru_{k}+B'\lambda_{k}-\beta_{k},
\label{1252}\\
\lambda_{k-1}&=&\frac{\partial H_{k}}{\partial x_{k}}=Qx_{k}+A'\lambda_{k}+K_{k}'\beta_{k}, \ \
\lambda_{N}=P_{N+1}x_{N+1},\label{1260}
\end{eqnarray}
which are consistent with \rf{2} and \rf{02}.
\end{proof}
A similar proof also be found in \cite{Bryson75}. Next, we derive the controller using Lemma \ref{l1}. Firstly, we consider $k=N$. From \rf{2} with $k=N$, we have that
\begin{eqnarray}
0&=&Ru_{N}+B'P_{N+1}(Ax_{N}+Bu_{N})-\beta_{N}\no\\
&=&(R+B'P_{N+1}B)u_{N}+B'P_{N+1}Ax_{N}-\beta_{N}.\no
\end{eqnarray}
Suppose that $R+B'P_{N+1}B$ is positive definite, it yields
\begin{eqnarray}
u_{N}=-(R+B'P_{N+1}B)^{-1}B'P_{N+1}Ax_{N}+(R+B'P_{N+1}B)^{-1}\beta_{N}.\label{12132}
\end{eqnarray}
Consider the linear form $u_{N}=K_{N}x_{N}$, combining \rf{12132}, we get
\begin{eqnarray}
[K_{N}+(R+B'P_{N+1}B)^{-1}B'P_{N+1}A]x_{N}-(R+B'P_{N+1}B)^{-1}\beta_{N}=0.\no
\end{eqnarray}
Due to the fact that the above relationship holds for any $x_{N}$, it deduces that
\begin{eqnarray}
K_{N}=-(R+B'P_{N+1}B)^{-1}B'P_{N+1}A, \ \ \beta_{N}=0,\label{12134}
\end{eqnarray}
i.e.,
\begin{eqnarray}
u_{N}=-(R+B'P_{N+1}B)^{-1}B'P_{N+1}Ax_{N}.\label{12133}
\end{eqnarray}
Based on this, from \rf{2}, we have
\begin{eqnarray}
\lambda_{N-1}=(Q+A'P_{N+1}A-A'P_{n+1}B(R+B'P_{n+1}B)^{-1}B'P_{n+1}A)x_{N}.\no
\end{eqnarray}
Following the line of $k=N$, by mathematical induction, we can deduce that the optimal solution is
\begin{eqnarray}
u_{k}=-(R+B'P_{k+1}B)^{-1}B'P_{k+1}Ax_{k}.\label{1288}
\end{eqnarray}
where $P_{k}=Q+A'P_{k+1}A-A'P_{k+1}B(R+B'P_{k+1}B)^{-1}B'P_{k+1}A$ with terminal value $P_{N+1}$. Moreover, the relationship between state and co-state is
\begin{eqnarray}
\lambda_{k-1}=P_{k}x_{k}.\label{12135}
\end{eqnarray}
\begin{remark}\label{r1}
 It is clear that the derived LQR controller  with constraint of current-state feedback is the same as the standard LQR without constraint. However, it would be not the case for the non-standard LQR problem like the Stackelberg game control to be studied in this paper.
\end{remark}
Inspired by the above constrained LQR problem, we study Problem 1 in the sequel.

\subsection{Results of the follower}
Following the sequential nature of a Stackelberg game, we first solve the follower's optimization problem. In this case, from \rf{f2.6}, system \rf{f2.1} and cost functional \rf{f2.3} can be rewritten as follows.
\begin{eqnarray}
x_{k+1}=(A+B_{2}\bar{K}_{k})x_{k}+B_{1}u^{f}_{k}+B_{2}\bar{G}_{k}x_{k-1}, \label{1283}
\end{eqnarray}
and
\begin{eqnarray}
J^{f}_{N}&=&\sum^{N}_{k=0}[x'_{k}(Q+\bar{K}_{k}'R_{2}\bar{K}_{k})x_{k}
+(u^{f}_{k})'R_{1}u^{f}_{k}+2x'_{k}\bar{K}_{k}'R_{2}\bar{G}_{k}x_{k-1}
 +x'_{k-1}\bar{G}_{k}'R_{2}\bar{G}_{k}x_{k-1}]\no\\
 &&
+x'_{N+1}P_{N+1}x_{N+1}.\label{1284}
\end{eqnarray}
Consider the follower's minimization problem with the linear constraint condition \rf{f2.5}, by introducing the Lagrange multipliers $\lambda^{f}_{k}$ and $\beta^{f}_{k}$, we define the following Hamiltonian functional
\begin{eqnarray}
 H^{f}_{k}&=&
x'_{k}Qx_{k}+(u^{f}_{k})'R_{1}u^{f}_{k}
 +x'_{k}\bar{K}_{k}'R_{2}\bar{K}_{k}x_{k}+2x'_{k}\bar{K}_{k}'R_{2}\bar{G}_{k}x_{k-1}
 +x'_{k-1}\bar{G}_{k}'R_{2}\bar{G}_{k}x_{k-1}\no\\
 &&+(\lambda^{f}_{k})'(Ax_{k}+B_{1}u^{f}_{k}+B_{2}\bar{K}_{k}x_{k}+B_{2}\bar{G}_{k}x_{k-1})
 +(\beta^{f}_{k})'(K_{k}x_{k}+G_{k}x_{k-1}-u^{f}_{k}).\label{f2.7}
\end{eqnarray}
Similar to the line for Lemma \ref{l1}, we have the following Lemma.
\begin{lemma}\label{l81}
Suppose the follower's optimization problem admits a solution, then, the controller $u^{f}_{k}$ satisfies
\begin{eqnarray}
0&=&R_{1}u^{f}_{k}+B_{1}'\lambda^{f}_{k}-\beta^{f}_{k},\label{f2.8}
\end{eqnarray}
where $\lambda^{f}_k$ is as
\begin{eqnarray}
\lambda^{f}_{k-1}&=&[Q+\bar{K}_{k}'R_{2}\bar{K}_{k}]x_{k}
+\bar{K}_{k}'R_{2}\bar{G}_{k}x_{k-1}+\bar{G}_{k+1}'R_{2}\bar{K}_{k+1}x_{k+1}
+\bar{G}_{k+1}'R_{2}\bar{G}_{k+1}x_{k}\no\\
&&+[A+B_{2}\bar{K}_{k}]'\lambda^{f}_{k}+\bar{G}_{k+1}'B_{2}'\lambda^{f}_{k+1}
+K_{k}'\beta^{f}_{k}+G_{k+1}'\beta^{f}_{k+1},\label{1264}\\
\lambda^{f}_{N}&=&P_{N+1}x_{N+1}\label{1265}
\end{eqnarray}
with $\bar{G}_{0}=\bar{G}_{N+1}=0$. And $\beta^{f}_{k}$ is arbitrary to be determined.
\end{lemma}
In the following, we shall solve the controller $u^{f}_{k}$ from Lemma \ref{l81}. To this end, we firstly define the following Riccati equation
\begin{eqnarray}
    P_{k}&=&Q+\bar{K}_{k}'R_{2}\bar{K}_{k}+(A+B_{2}\bar{K}_{k})'P_{k+1}(A+B_{2}\bar{K}_{k})
+(A+B_{2}\bar{K}_{k})'S_{k+1}'+S_{k+1}(A+B_{2}\bar{K}_{k})\no\\
&&
-(M^{1}_{k}+M^{B}_{k}\bar{K}_{k}+B_{1}'S_{k+1}')'(\Gamma^{f}_{k})^{-1}
(M^{1}_{k}+M^{B}_{k}\bar{K}_{k}+B_{1}'S_{k+1}')
+\bar{G}_{k+1}'\Delta_{k+1}\bar{G}_{k+1},\label{2.9}\\
S_{k}
&=&\bar{G}_{k}'\Delta_{k}\bar{K}_{k}+\bar{G}_{k}'\mathcal{M}^{2}_{k},\label{2.10}
\end{eqnarray}
where
\begin{eqnarray}
M^{i}_{k}&=&B_{i}'P_{k+1}A,\ \ i=1,2, \label{2.111}\\
M^{B}_{k}&=&B_{1}'P_{k+1}B_{2},\\
\Gamma^{f}_{k}&=&R_{1}+B_{1}'P_{k+1}B_{1},\label{2.112}\\
\Delta_{k}&=&R_{2}+ B_{2}'P_{k+1}B_{2}-(M^{B}_{k})'(\Gamma^{f}_{k})^{-1}M^{B}_{k}, \label{2.113}\\
\mathcal{M}^{2}_{k}&=&
M^{2}_{k}-(M^{B}_{k})'(\Gamma^{f}_{k})^{-1}M^{1}_{k}+[-B_{1}(\Gamma^{f}_{k})^{-1}M^{B}_{k}
+B_{2}]'S_{k+1}\label{2.11}
\end{eqnarray}
with terminal values $P_{N+1}, S_{N+1}=0$. In the above, $\bar{K}_{k}$ and $\bar{G}_{k}$ are the gain matrices of the leader as defined  in \rf{f2.6}.

Accordingly, we give the main result of the follower's optimization problem.
\begin{theorem}\label{t1}
If Riccati equations (\ref{2.9}) and (\ref{2.10}) have a solution such that $\Gamma^{f}_{k}>0$, then the follower's optimal controller is
\begin{eqnarray}
u^{f}_{k}&=&-(\Gamma^{f}_{k})^{-1}[M^{1}_{k}+M^{B}_{k}\bar{K}_{k}+B_{1}'S_{k+1}']x_{k}
-(\Gamma^{f}_{k})^{-1}M^{B}_{k}\bar{G}_{k}x_{k-1},\label{2.12}
\end{eqnarray}
and the co-state $\lambda_{k-1}$ is given as follows.
\begin{eqnarray}
\lambda_{k-1}&=&P_{k}x_{k}+S_{k}x_{k-1}.\label{2.13}
\end{eqnarray}
Moreover, the follower's optimal cost value is
\begin{eqnarray}
J^{f\ast}=x_{0}'P_{0}x_{0}.\label{f2.21}
\end{eqnarray}
\end{theorem}
\begin{proof}
We adopt the backward iteration method to find the optimal controller.
When $k=N$, from \rf{f2.8}, we get
\begin{eqnarray}
0&=&R_{1}u^{f}_{N}+B_{1}'P_{N+1}(Ax_{N}+B_{1}u^{f}_{N}+B_{2}\bar{K}_{N}x_{N}
+B_{2}\bar{G}_{N}x_{N-1})-\beta^{f}_{N}\no\\
&=&(R_{1}+B_{1}'P_{N+1}B_{1})u^{f}_{N}+B_{1}'P_{N+1}(A+B_{2}\bar{K}_{N})x_{N}
+B_{1}'P_{N+1}B_{2}\bar{G}_{N}x_{N-1}-\beta^{f}_{N},\label{f2.16}
\end{eqnarray}
i.e.,
\begin{eqnarray}
u^{f}_{N}&=&-(\Gamma^{f}_{N})^{-1}[M^{1}_{N}+M^{B}_{N}\bar{K}_{N}]x_{N}
-(\Gamma^{f}_{N})^{-1}M^{B}_{N}\bar{G}_{N}x_{N-1}
+(\Gamma^{f}_{N})^{-1}\beta^{f}_{N}.\label{f2.17}
\end{eqnarray}
From the linear closed-loop form, i.e., $u^{f}_{N}=K_{N}x_{N}+G_{N}x_{N-1}$, we derive that
\begin{eqnarray}
[K_{N}+(\Gamma^{f}_{N})^{-1}(M^{1}_{N}+M^{B}_{N}\bar{K}_{N})]x_{N}
+[G_{N}+(\Gamma^{f}_{N})^{-1}M^{B}_{N}\bar{G}_{N}]x_{N-1}
-(\Gamma^{f}_{N})^{-1}\beta^{f}_{N}=0.\label{f2.17}
\end{eqnarray}
Note that \rf{f2.17} holds for any $\beta^{f}_{N}$ 
, it yields that
\begin{eqnarray}
[K_{N}+(\Gamma^{f}_{N})^{-1}(M^{1}_{N}+M^{B}_{N}\bar{K}_{N})]x_{N}
+[G_{N}+(\Gamma^{f}_{N})^{-1}M^{B}_{N}\bar{G}_{N}]x_{N-1}=0.\no
\end{eqnarray}
That is, the optimal controller is given by
\begin{eqnarray}
u^{f}_{N}&=&K_{N}x_{N}+G_{N}x_{N-1}\no\\
&=&-(\Gamma^{f}_{N})^{-1}[M^{1}_{N}+M^{B}_{N}\bar{K}_{N}]x_{N}
-(\Gamma^{f}_{N})^{-1}M^{B}_{N}\bar{G}_{N}x_{N-1}.\label{1257}
\end{eqnarray}
In this case, consider \rf{f2.17} again, it deduces that $(\Gamma^{f}_{N})^{-1}\beta^{f}_{N}=0$. Due to the invertibility of $\Gamma^{f}_{N}$, obviously, it yields that $\beta^{f}_{N}=0$.

From \rf{1264}, we have
\begin{eqnarray}
\lambda^{f}_{N-1}&=&[Q+\bar{K}_{N}'R_{2}\bar{K}_{N}]x_{N}
+\bar{K}_{N}'R_{2}\bar{G}_{N}x_{N-1}
+[A+B_{2}\bar{K}_{N}]'P_{N+1}[(A+B_{2}\bar{K}_{N})x_{N}\no\\
&&-B_{1}(\Gamma^{f}_{N})^{-1}
(M^{1}_{N}+M^{B}_{N}\bar{K}_{N})x_{N}-B_{1}(\Gamma^{f}_{N})^{-1}M^{B}_{N}\bar{G}_{N}x_{N-1}
+B_{2}\bar{G}_{N}x_{N-1}]\no\\
&=&\{Q+\bar{K}_{N}'R_{2}\bar{K}_{N}+(A+B_{2}\bar{K}_{N})'P_{N+1}(A+B_{2}\bar{K}_{N})
-(M^{1}_{N}+M^{B}_{N}\bar{K}_{N})'(\Gamma^{f}_{N})^{-1}\no\\
&&\times(M^{1}_{N}+M^{B}_{N}\bar{K}_{N})\}x_{N}
+\{\bar{K}_{N}'[R_{2}+B_{2}'P_{N+1}B_{2}-(M^{B}_{N})'(\Gamma^{f}_{N})^{-1}M^{B}_{N}]\no\\
&&
-(M^{1}_{N})'(\Gamma^{f}_{N})^{-1}M^{B}_{N}
+(M^{2}_{N})'\}\bar{G}_{N}x_{N-1},\no\\
&=&P_{N}x_{N}+S_{N}x_{N-1}.\label{f2.18}
\end{eqnarray}

For any $n$ with $0\leq n\leq N$, we assume that $\Gamma^{f}_{k}>0$ and the follower's optimal controller $u^{f}_{k}$ and co-state $\lambda^{f}_{k-1}$ are respectively given as in \rf{2.12} and \rf{2.13} for $k>n$.  Now we show $u^f_{k}$ and $\lambda^{f}_{k-1}$ have the same  expression as \rf{2.12} and \rf{2.13} for $k=n$, respectively.  For $u^{f}_{n}$, from \rf{f2.8}, we have that
\begin{eqnarray}
0&=&R_{1}u^{f}_{n}+B_{1}'P_{n+1}[(A+B_{2}\bar{K}_{n})x_{n}+B_{1}u^{f}_{n}
+B_{2}\bar{G}_{n}x_{n-1}]+B_{1}'S_{n+1}x_{n}-\beta^{f}_{n}\no\\
&=&[R_{1}+B_{1}'P_{n+1}B_{1}]u^{f}_{n}+[(M^{1}_{n}+M^{B}_{n}\bar{K}_{n})+B_{1}'S_{n+1}']x_{n}
+M^{B}_{n}\bar{G}_{n}x_{n-1}-\beta^{f}_{n},\no
\end{eqnarray}
it yields that
\begin{eqnarray}
u^{f}_{n}&=&-(\Gamma^{f}_{n})^{-1}[M^{1}_{n}+M^{B}_{n}\bar{K}_{n}+B_{1}'S_{n+1}']x_{n}
-(\Gamma^{f}_{n})^{-1}M^{B}_{n}\bar{G}_{n}x_{n-1}
+(\Gamma^{f}_{n})^{-1}\beta^{f}_{n},\no
\end{eqnarray}
combining that linear form $u^{f}_{n}=K_{n}x_{n}+G_{n}x_{n-1}$, hence, we have that
\begin{eqnarray}
[K_{n}+(\Gamma^{f}_{n})^{-1}(M^{1}_{n}+M^{B}_{n}\bar{K}_{n}+B_{1}'S_{n+1}')]x_{n}
+[G_{n}+(\Gamma^{f}_{n})^{-1}M^{B}_{n}\bar{G}_{n}]x_{n-1}
-(\Gamma^{f}_{n})^{-1}\beta^{f}_{n}=0,\label{121313}
\end{eqnarray}
Note that \rf{121313} holds for any $\beta^{f}_{n}$, it deduces that
\begin{eqnarray}
[K_{n}+(\Gamma^{f}_{n})^{-1}(M^{1}_{n}+M^{B}_{n}\bar{K}_{n}+B_{1}'S_{n+1}')]x_{n}
+[G_{n}+(\Gamma^{f}_{n})^{-1}M^{B}_{n}\bar{G}_{n}]x_{n-1}=0.\no
\end{eqnarray}
That is, the optimal controller is given by
\begin{eqnarray}
u^{f}_{n}&=&K_{n}x_{n}+G_{n}x_{n-1}\no\\
&=&-(\Gamma^{f}_{n})^{-1}(M^{1}_{n}+M^{B}_{n}\bar{K}_{n}+B_{1}'S_{n+1}')x_{n}
-(\Gamma^{f}_{n})^{-1}M^{B}_{n}\bar{G}_{n}x_{n-1}.\label{f2.20}
\end{eqnarray}
In this case, from \rf{121313}, due to the invertibility of $\Gamma^{f}_{n}$, it yields that $\beta^{f}_{n}=0$. Thus, by induction it is easy to know  the controller  $u_k^f$  for any $0\leq k\leq N$ has the form as \rf{2.12}.

Next, we derive the expression of $\lambda^{f}_{n-1}$. In view of \rf{1264}, it follows
\begin{small}\begin{eqnarray}
\lambda^{f}_{n-1}&=&(Q+\bar{K}_{n}'R_{2}\bar{K}_{n})x_{n}
+\bar{K}_{n}'R_{2}\bar{G}_{n}x_{n-1}+\bar{G}_{n+1}'R_{2}\bar{K}_{n+1}x_{n+1}
\no\\
&&+\bar{G}_{n+1}'R_{2}\bar{G}_{n+1}x_{n}+(A+B_{2}\bar{K}_{n})'\lambda^{f}_{n}
+\bar{G}_{n+1}'B_{2}'\lambda^{f}_{n+1}\no\\
&=&(Q+\bar{K}_{n}'R_{2}\bar{K}_{n})x_{n}
+\bar{K}_{n}'R_{2}\bar{G}_{n}x_{n-1}+\bar{G}_{n+1}'R_{2}\bar{K}_{n+1}x_{n+1}
\no\\
&&+\bar{G}_{n+1}'R_{2}\bar{G}_{n+1}x_{n}+(A+B_{2}\bar{K}_{n})'P_{n+1}x_{n+1}
+(A+B_{2}\bar{K}_{n})'S_{n+1}x_{n}\no\\
&&
+\bar{G}_{n+1}'B_{2}'P_{n+2}[(A+B_{2}\bar{K}_{n+1})x_{n+1}+B_{1}u^{f}_{n+1}
+B_{2}\bar{G}_{n+1}x_{n}]+\bar{G}_{n+1}'B_{2}'S_{n+2}x_{n+1}\no\\
&=&\{Q+\bar{K}_{n}'R_{2}\bar{K}_{n}+\bar{G}_{n+1}'R_{2}\bar{G}_{n+1}
+(A+B_{2}\bar{K}_{n})'S_{n+1}+\bar{G}_{n+1}'B_{2}'P_{n+2}B_{2}\bar{G}_{n+1}\no\\
&&
-\bar{G}_{n+1}'B_{2}'P_{n+1}B_{1}(\Gamma^{f}_{n+1})^{-1}M^{B}_{n+1}\bar{G}_{n+1}\}x_{n}
+\bar{K}_{n}'R_{2}\bar{G}_{n}x_{n-1}
+\{\bar{G}_{n+1}'R_{2}\bar{K}_{n+1}\no\\
&&+(A+B_{2}\bar{K}_{n})'P_{n+1}+\bar{G}_{n+1}'B_{2}'P_{n+2}(A+B_{2}\bar{K}_{n+1})
-\bar{G}_{n+1}'B_{2}'P_{n+2}B_{1}(\Gamma^{f}_{n+1})^{-1}\no\\
&&\times[M^{1}_{n+1}+M^{B}_{n+1}\bar{K}_{n+1}]+\bar{G}_{n+1}'B_{2}'S_{n+2}\}x_{n+1}
\no\\
&=&\{Q+\bar{K}_{n}'R_{2}\bar{K}_{n}+\bar{G}_{n+1}'R_{2}\bar{G}_{n+1}
+(A+B_{2}\bar{K}_{n})'S_{n+1}+\bar{G}_{n+1}'B_{2}'P_{n+2}B_{2}\bar{G}_{n+1}\no\\
&&-\bar{G}_{n+1}'B_{2}'P_{n+1}B_{1}(\Gamma^{f}_{n+1})^{-1}M^{B}_{n+1}\bar{G}_{n+1}\}x_{n}
+\bar{K}_{n}'R_{2}\bar{G}_{n}x_{n-1}\no\\
&&
+\{\bar{G}_{n+1}'R_{2}\bar{K}_{n+1}+(A+B_{2}\bar{K}_{n})'P_{n+1}
+\bar{G}_{n+1}'B_{2}'P_{n+2}(A+B_{2}\bar{K}_{n+1})\no\\
&&
-\bar{G}_{n+1}'B_{2}'P_{n+2}B_{1}(\Gamma^{f}_{n+1})^{-1}[M^{1}_{n+1}+M^{B}_{n+1}\bar{K}_{n+1}]
+\bar{G}_{n+1}'B_{2}'S_{n+2}\}
\{(A+B_{2}\bar{K}_{n})x_{n}\no\\
&&-B_{1}(\Gamma^{f}_{n})^{-1}[(M^{1}_{n}+M^{B}_{n}\bar{K}_{n})
+B_{1}'S_{n+1}']x_{n}
-B_{1}(\Gamma^{f}_{n})^{-1}M^{B}_{n}\bar{G}_{n}x_{n-1}
+B_{2}\bar{G}_{n}x_{n-1}\}
\no\\
&=&\Big\{Q+\bar{K}_{n}'R_{2}\bar{K}_{n}+\bar{G}_{n+1}'R_{2}\bar{G}_{n+1}
+(A+B_{2}\bar{K}_{n})'S_{n+1}+\bar{G}_{n+1}'B_{2}'P_{n+2}B_{2}\bar{G}_{n+1}\no\\
&&-\bar{G}_{n+1}'(M^{B}_{n+1})'(\Gamma^{f}_{n+1})^{-1}M^{B}_{n+1}\bar{G}_{n+1}
+\{\bar{G}_{n+1}'R_{2}\bar{K}_{n+1}+(A+B_{2}\bar{K}_{n})'P_{n+1}\no\\
&&+\bar{G}_{n+1}'B_{2}'P_{n+2}(A+B_{2}\bar{K}_{n+1})
-\bar{G}_{n+1}'B_{2}'P_{n+2}B_{1}(\Gamma^{f}_{n+1})^{-1}[M^{1}_{n+1}+M^{B}_{n+1}\bar{K}_{n+1}]
\no\\
&&+\bar{G}_{n+1}'B_{2}'S_{n+2}\}\{(A+B_{2}\bar{K}_{n})-B_{1}(\Gamma^{f}_{n})^{-1}[(M^{1}_{n}+M^{B}_{n}
\bar{K}_{n})+B_{1}'S_{n+1}']\}
\Big\}x_{n}\no\\
&&
+\Big\{\bar{K}_{n}'R_{2}\bar{G}_{n}+\{\bar{G}_{n+1}'R_{2}\bar{K}_{n+1}+(A+B_{2}\bar{K}_{n})'
P_{n+1}+\bar{G}_{n+1}'B_{2}'P_{n+2}(A+B_{2}\bar{K}_{n+1})\no\\
&&
-\bar{G}_{n+1}'B_{2}'P_{n+2}B_{1}(\Gamma^{f}_{n+1})^{-1}[M^{1}_{n+1}+M^{B}_{n+1}\bar{K}_{n+1}]
+\bar{G}_{n+1}'B_{2}'S_{n+2}\}
[-B_{1}(\Gamma^{f}_{n})^{-1}M^{B}_{n}\bar{G}_{n}
+B_{2}\bar{G}_{n}]\Big\}x_{n-1}\no\\
&=&\Big\{Q+\bar{K}_{n}'R_{2}\bar{K}_{n}+\bar{G}_{n+1}'[R_{2}+B_{2}'P_{n+2}B_{2}
-(M^{B}_{n+1})'(\Gamma^{f}_{n+1})^{-1}M^{B}_{n+1}]\bar{G}_{n+1}
+(A+B_{2}\bar{K}_{n})'S_{n+1}\no\\
&&+\{\bar{G}_{n+1}'[R_{2}+B_{2}'P_{n+2}B_{2}
-(M^{B}_{n+1})'(\Gamma^{f}_{n+1})^{-1}M^{B}_{n+1}]\bar{K}_{n+1}
+\bar{G}_{n+1}'[M^{2}_{n+1}-(M^{B}_{n+1})'(\Gamma^{f}_{n+1})^{-1}M^{1}_{n+1}\no\\
&&+B_{2}'S_{n+2}]+(A+B_{2}\bar{K}_{n})'P_{n+1}\}
\Big[(A+B_{2}\bar{K}_{n})-B_{1}(\Gamma^{f}_{n})^{-1}[(M^{1}_{n}+M^{B}_{n}
\bar{K}_{n})+B_{1}'S_{n+1}]\Big]
\Big\}x_{n}\no\\
&&+\Big\{\bar{K}_{n}'R_{2}\bar{G}_{n}+\{\bar{G}_{n+1}'R_{2}\bar{K}_{n+1}+(A+B_{2}\bar{K}_{n})'
P_{n+1}+\bar{G}_{n+1}'B_{2}'P_{n+2}(A+B_{2}\bar{K}_{n+1})\no\\
&&
-\bar{G}_{n+1}'B_{2}'P_{n+2}B_{1}(\Gamma^{f}_{n+1})^{-1}[M^{1}_{n+1}+M^{B}_{n+1}\bar{K}_{n+1}]
+\bar{G}_{n+1}'B_{2}'S_{n+2}\}
[-B_{1}(\Gamma^{f}_{n})^{-1}M^{B}_{n}\bar{G}_{n}
+B_{2}\bar{G}_{n}]\Big\}x_{n-1}\no
\end{eqnarray}\end{small}
\begin{small}\begin{eqnarray}
&=&\Big\{Q+\bar{K}_{n}'R_{2}\bar{K}_{n}+\bar{G}_{n+1}'\Delta_{n+1}\bar{G}_{n+1}
+(A+B_{2}\bar{K}_{n})'S_{n+1}'\no\\
&&+\Big[S_{n+1}+(A+B_{2}\bar{K}_{n})'P_{n+1}\Big]
\Big[(A+B_{2}\bar{K}_{n})-B_{1}(\Gamma^{f}_{n})^{-1}B_{1}'[S_{n+1}'+P_{n+1}(A+B_{2}\bar{K}_{n})'
]\Big]\Big\}x_{n}\no\\
&&
+\Big\{\bar{K}_{n}'R_{2}\bar{G}_{n}+\{\bar{G}_{n+1}'R_{2}\bar{K}_{n+1}+(A+B_{2}\bar{K}_{n})'
P_{n+1}+\bar{G}_{n+1}'B_{2}'P_{n+1}(A+B_{2}\bar{K}_{n+1})\no\\
&&
-\bar{G}_{n+1}'B_{2}'P_{n+1}B_{1}(\Gamma^{f}_{n+1})^{-1}[M^{1}_{n+1}+M^{B}_{n+1}\bar{K}_{n+1}]\}
[-B_{1}(\Gamma^{f}_{n})^{-1}M^{B}_{n}\bar{G}_{n}
+B_{2}\bar{G}_{n}]\Big\}x_{n-1}\no\\
&=&P_{n}x_{n}+S_{n}x_{n-1}.\no
\end{eqnarray}\end{small}\\
Thus, by induction it is easy to know that the co-state $\lambda_k^f$ for any $0\leq k\leq N$ has the form as \rf{2.13}. Now the proof of Theorem \ref{t1} is completed.

\end{proof}
\begin{remark}\label{r2}
Note that from the above results, we only give the expression of the optimal controller \rf{2.12}. However, we don't show the form of parameters $K_{k}$ and $G_{k}$. We can illustrate it with the derivation process of $u^{f}_{k}$. From the relevance between $x_{k}$ and $x_{k-1}$,  it cannot directly deduce that
$K_{k}+(\Gamma^{f}_{k})^{-1}(M^{1}_{k}+M^{B}_{k}\bar{K}_{k})=0$ and $G_{k}+(\Gamma^{f}_{k})^{-1}M^{B}_{k}\bar{G}_{k}=0$ based on \rf{f2.17}. In other words, we can not derive the gain matrices of the follower as $K_{k}=-(\Gamma^{f}_{k})^{-1}(M^{1}_{k}+M^{B}_{k}\bar{K}_{k})$ and $G_{k}=-(\Gamma^{f}_{k})^{-1}M^{B}_{k}\bar{G}_{k}$  though the optimal controller of the follower is given by \rf{2.12}. Of course, this will not affect the calculation of the optimal controller.
\end{remark}

\subsection{Results of the leader}
According to the discussion of follower's optimization problem, system \rf{f2.1} and leader's cost functional \rf{f2.4} are respectively rewritten as follows:
\begin{eqnarray}
x_{k+1}&=&\{A-B_{1}(\Gamma^{f}_{k})^{-1}[M^{1}_{k}+M^{B}_{k}\bar{K}_{k}+B_{1}'S_{k+1}']\}x_{k}
-B_{1}(\Gamma^{f}_{k})^{-1}M^{B}_{k}\bar{G}_{k}x_{k-1}+B_{2}u^{l}_{k}\no\\
&=&[A-B_{1}(\Gamma^{f}_{k})^{-1}(M^{1}_{k}+B_{1}'S_{k+1}')]x_{k}+
[B_{2}-B_{1}(\Gamma^{f}_{k})^{-1}M^{B}_{k}]u^{l}_{k}, \label{1286}
\end{eqnarray}
and
\begin{eqnarray}
 J^{l}_{N}&=&\sum^{N}_{k=0}\Big\{x'_{k}\Big[\bar{Q}+(M^{1}_{k}+B_{1}'S_{k+1}')'(\Gamma^{f}_{k})^{-1}
\bar{R}_{1}(\Gamma^{f}_{k})^{-1}(M^{1}_{k}+B_{1}'S_{k+1}')\Big]x_{k}
+2x'_{k}(M^{1}_{k}+B_{1}'S_{k+1}')'(\Gamma^{f}_{k})^{-1}\no\\
&&
\times\bar{R}_{1}(\Gamma^{f}_{k})^{-1}M^{B}_{k}u^{l}_{k}
+(u^{l}_{k})'\Big[\bar{R}_{2}+(M^{B}_{k})'(\Gamma^{f}_{k})^{-1}\bar{R}_{1}(\Gamma^{f}_{k})^{-1}
M^{B}_{k}\Big]u^{l}_{k}\Big\}+x'_{N+1}\bar{P}_{N+1}x_{N+1}.\label{1287}
\end{eqnarray}
To solve the leader's optimization problem with constraint condition \rf{f2.6}, we define the following functional with Lagrange multipliers $\lambda^{l}_{k}$ and $\beta^{l}_{k}$.
\begin{eqnarray}
  H^{l}_{k}&=&x'_{k}\Big[\bar{Q}+(M^{1}_{k}+B_{1}'S_{k+1}')'(\Gamma^{f}_{k})^{-1}
\bar{R}_{1}(\Gamma^{f}_{k})^{-1}(M^{1}_{k}+B_{1}'S_{k+1}')\Big]x_{k}\no\\
&&+2x'_{k}(M^{1}_{k}+B_{1}'S_{k+1}')'(\Gamma^{f}_{k})^{-1}
\bar{R}_{1}(\Gamma^{f}_{k})^{-1}M^{B}_{k}u^{l}_{k}
+(u^{l}_{k})'\Big[\bar{R}_{2}+(M^{B}_{k})'(\Gamma^{f}_{k})^{-1}\bar{R}_{1}(\Gamma^{f}_{k})^{-1}
M^{B}_{k}\Big]u^{l}_{k}\no\\
&&
 +(\lambda^{l}_{k})'\{[A-B_{1}(\Gamma^{f}_{k})^{-1}(M^{1}_{k}+B_{1}'S_{k+1}')]x_{k}+
[B_{2}-B_{1}(\Gamma^{f}_{k})^{-1}M^{B}_{k}]u^{l}_{k}\}\no\\
&&+(\beta^{l}_{k})'[\bar{K}_{k}x_{k}+\bar{G}_{k}x_{k-1}-u^{l}_{k}].\label{f2.30}
\end{eqnarray}
According to Lemma \ref{l1}, we can obtain the following result.
\begin{lemma}\label{l4}
Suppose the leader's optimization problem admits a solution, then, the leader's controller $u^{l}_{k}$ satisfies
\begin{eqnarray}
0&=&(M^{B}_{k})'(\Gamma^{f}_{k})^{-1}\bar{R}_{1}
(\Gamma^{f}_{k})^{-1}(M^{1}_{k}+B_{1}'S_{k+1}')x_{k}
+\Big[\bar{R}_{2}+(M^{B}_{k})'(\Gamma^{f}_{k})^{-1}\bar{R}_{1}(\Gamma^{f}_{k})^{-1}
M^{B}_{k}\Big]u^{l}_{k}\no\\
&& +[B_{2}-B_{1}(\Gamma^{f}_{k})^{-1}M^{B}_{k}]'\lambda^{l}_{k}-\beta^{l}_{k},\label{f2.31}
\end{eqnarray}
where $\lambda^{l}_k$ is as
\begin{eqnarray}
\lambda^{l}_{k-1}&=&\Big[\bar{Q}+(M^{1}_{k}+B_{1}'S_{k+1}')'(\Gamma^{f}_{k})^{-1}
\bar{R}_{1}(\Gamma^{f}_{k})^{-1}(M^{1}_{k}+B_{1}'S_{k+1}')\Big]x_{k}
+(M^{1}_{k}+B_{1}'S_{k+1}')'(\Gamma^{f}_{k})^{-1}\no\\
&&\times\bar{R}_{1}(\Gamma^{f}_{k})^{-1}M^{B}_{k}u^{l}_{k}
+[A-B_{1}(\Gamma^{f}_{k})^{-1}(M^{1}_{k}+B_{1}'S_{k+1}')]'\lambda^{l}_{k}
+\bar{K}_{k}'\beta^{l}_{k}+\bar{G}_{k+1}\beta^{l}_{k+1},\label{f2.311}\\
\lambda^{l}_{N}&=&\bar{P}_{N+1}x_{N+1},\label{f2.312}
\end{eqnarray}
with $\bar{G}_{0}=\bar{G}_{N+1}=0$. And $\beta^{l}_{k}$ is arbitrary to be determined.
\end{lemma}
Next, we shall solve the controller $u^{l}_{k}$ from Lemma \ref{l4}. To this end, we firstly define the following Riccati equation
\begin{eqnarray}
\bar{P}_{k}&=&\bar{Q}+(M^{1}_{k}+B_{1}'S_{k+1}')'(\Gamma^{f}_{k})^{-1}
\bar{R}_{1}(\Gamma^{f}_{k})^{-1}(M^{1}_{k}+B_{1}'S_{k+1}')
-(\mathcal{M}^{l}_{k})'(\Gamma^{l}_{k})^{-1}\mathcal{M}^{l}_{k}\no\\
&&+[A-B_{1}(\Gamma^{f}_{k})^{-1}(M^{1}_{k}+B_{1}'S_{k+1}')]'\bar{P}_{k+1}
[A-B_{1}(\Gamma^{f}_{k})^{-1}(M^{1}_{k}+B_{1}'S_{k+1}')],\label{f2.32}
\end{eqnarray}
where
\begin{eqnarray}
 \Gamma^{l}_{k}&=&\bar{R}_{2}+(M^{B}_{k})'(\Gamma^{f}_{k})^{-1}\bar{R}_{1}(\Gamma^{f}_{k})^{-1}
M^{B}_{k}+[B_{2}-B_{1}(\Gamma^{f}_{k})^{-1}M^{B}_{k}]'\bar{P}_{k+1}
[B_{2}-B_{1}(\Gamma^{f}_{k})^{-1}M^{B}_{k}],\label{f2.331}\\
\mathcal{M}^{l}_{k}&=&(M^{B}_{k})'(\Gamma^{f}_{k})^{-1}\bar{R}_{1}
(\Gamma^{f}_{k})^{-1}(M^{1}_{k}+B_{1}'S_{k+1}')+[B_{2}-B_{1}(\Gamma^{f}_{k})^{-1}
M^{B}_{k}]'\bar{P}_{k+1}\no\\
&&\times[A-B_{1}(\Gamma^{f}_{k})^{-1}(M^{1}_{k}+B_{1}'S_{k+1}')].\label{f2.332}
\end{eqnarray}

Now, we give the results of the leader's optimization problem. It should be pointed out that the discussion of leader's optimization problem is based on the solvablility of the follower's optimization problem.
\begin{theorem}\label{t2}
If the Riccati equation (\ref{f2.32}) have a solution such that $\Gamma^{l}_{k}>0$, then, the leader's optimal linear strategy is
\begin{eqnarray}
u^{l}_{k}&=&-(\Gamma^{l}_{k})^{-1}\mathcal{M}^{l}_{k}x_{k},\label{f2.33}
\end{eqnarray}
and the relationship between state and co-state is
\begin{eqnarray}
\lambda^{l}_{k-1}&=&\bar{P}_{k}x_{k}.\label{f2.34}
\end{eqnarray}
Moreover, the leader's optimal cost value is given by
\begin{eqnarray}
J^{l\ast}&=&x_{0}'\bar{P}_{0}x_{0}.\label{f2.35}
\end{eqnarray}
\end{theorem}

\begin{proof}
Using the backward iteration method, from \rf{f2.31} with $k=N$, we obtain
\begin{eqnarray}
0&=&(M^{B}_{N})'(\Gamma^{f}_{N})^{-1}\bar{R}_{1}
(\Gamma^{f}_{N})^{-1}M^{1}_{N}x_{N}
+\Big[\bar{R}_{2}+(M^{B}_{N})'(\Gamma^{f}_{N})^{-1}\bar{R}_{1}(\Gamma^{f}_{N})^{-1}
M^{B}_{N}\Big]u^{l}_{N}\no\\
&&
 +[B_{2}-B_{1}(\Gamma^{f}_{N})^{-1}M^{B}_{N}]'\bar{P}_{N+1}
\{[A-B_{1}(\Gamma^{f}_{N})^{-1}M^{1}_{N}]x_{N}+
[B_{2}-B_{1}(\Gamma^{f}_{N})^{-1}M^{B}_{N}]u^{l}_{N}\}-\beta^{l}_{N}\no\\
&=&\Big\{\bar{R}_{2}+(M^{B}_{N})'(\Gamma^{f}_{N})^{-1}\bar{R}_{1}(\Gamma^{f}_{N})^{-1}
M^{B}_{N}+[B_{2}-B_{1}(\Gamma^{f}_{N})^{-1}M^{B}_{N}]'\bar{P}_{N+1}
[B_{2}-B_{1}(\Gamma^{f}_{N})^{-1}M^{B}_{N}]\Big\}u^{l}_{N}\no\\
&&
+\Big\{(M^{B}_{N})'(\Gamma^{f}_{N})^{-1}\bar{R}_{1}
(\Gamma^{f}_{N})^{-1}M^{1}_{N}+[B_{2}-B_{1}(\Gamma^{f}_{N})^{-1}M^{B}_{N}]'\bar{P}_{N+1}
[A-B_{1}(\Gamma^{f}_{N})^{-1}M^{1}_{N}]\Big\}x_{N}-\beta^{l}_{N},\no
\end{eqnarray}
i.e.,
\begin{eqnarray}
u^{l}_{N}&=&-(\Gamma^{l}_{N})^{-1}\mathcal{M}^{l}_{N}x_{N}
+(\Gamma^{l}_{N})^{-1}\beta^{l}_{N}.\label{f2.37}
\end{eqnarray}
From the linear closed-loop form $\bar{K}_{N}x_{N}+\bar{G}_{N}x_{N-1}$, we derive that
\begin{eqnarray}
[\bar{K}_{N}+(\Gamma^{l}_{N})^{-1}\mathcal{M}^{l}_{N}]x_{N}
+\bar{G}_{N}x_{N-1}-(\Gamma^{l}_{N})^{-1}\beta^{l}_{N}=0.\label{1259}
\end{eqnarray}
Note that \rf{1259} holds for any $\beta^{l}_{N}$, it yields that
\begin{eqnarray}
[\bar{K}_{N}+(\Gamma^{l}_{N})^{-1}\mathcal{M}^{l}_{N}]x_{N}+\bar{G}_{N}x_{N-1}=0.\no
\end{eqnarray}
That is, the optimal controller is given by
\begin{eqnarray}
u^{f}_{N}&=&\bar{K}_{N}x_{N}+\bar{G}_{N}x_{N-1}\no\\
&=&-(\Gamma^{l}_{N})^{-1}\mathcal{M}^{l}_{N}x_{N}.\label{121314}
\end{eqnarray}
On the other hand, consider \rf{1259} again, due to the invertibility of $\Gamma^{l}_{N}$, it yields that $\beta^{l}_{N}=0$.

For $\lambda^{l}_{N-1}$, from \rf{f2.311}, we have
\begin{eqnarray}
\lambda^{l}_{N-1}&=&\Big[\bar{Q}+(M^{1}_{N})'(\Gamma^{f}_{N})^{-1}
\bar{R}_{1}(\Gamma^{f}_{N})^{-1}(M^{1}_{N})\Big]x_{N}
+(M^{1}_{N})'(\Gamma^{f}_{N})^{-1}
\bar{R}_{1}(\Gamma^{f}_{N})^{-1}M^{B}_{N}u^{l}_{N}\no\\
&&+[A-B_{1}(\Gamma^{f}_{N})^{-1}(M^{1}_{N})]'\bar{P}_{N+1}
\{[A-B_{1}(\Gamma^{f}_{N})^{-1}M^{1}_{N}]x_{N}+
[B_{2}-B_{1}(\Gamma^{f}_{N})^{-1}M^{B}_{N}]u^{l}_{N}\}\no\\
&=&\Big\{\bar{Q}+(M^{1}_{N})'(\Gamma^{f}_{N})^{-1}
\bar{R}_{1}(\Gamma^{f}_{N})^{-1}(M^{1}_{N})+[A-B_{1}(\Gamma^{f}_{N})^{-1}(M^{1}_{N})]'\bar{P}_{N+1}
[A-B_{1}(\Gamma^{f}_{N})^{-1}M^{1}_{N}]\Big\}x_{N}\no\\
&&+\Big\{(M^{1}_{N})'(\Gamma^{f}_{N})^{-1}
\bar{R}_{1}(\Gamma^{f}_{N})^{-1}M^{B}_{N}
+[A-B_{1}(\Gamma^{f}_{N})^{-1}(M^{1}_{N})]'\bar{P}_{N+1}
[B_{2}-B_{1}(\Gamma^{f}_{N})^{-1}M^{B}_{N}]\Big\}u^{l}_{N}\no\\
&=&\bar{P}_{N}x_{N}.\label{f2.38}
\end{eqnarray}

For any $n$ with $0\leq n\leq N$, we assume that $\Gamma^{l}_{k}>0$ and the optimal follower's controller $u^{l}_{k}$ and co-state $\lambda^{l}_{k-1}$ are respectively given as in \rf{f2.33} and \rf{f2.34} for $k>n$. Next, we show that $u^{l}_{k}$ and $\lambda^{l}_{k-1}$ have the same form as \rf{f2.33} and \rf{f2.34} for $k=n$, respectively. For $u^{l}_{n}$, from \rf{f2.31}, we have that
\begin{eqnarray}
0&=&(M^{B}_{n})'(\Gamma^{f}_{n})^{-1}\bar{R}_{1}
(\Gamma^{f}_{n})^{-1}(M^{1}_{n}+B_{1}'S_{k+1}')x_{n}
+\Big[\bar{R}_{2}+(M^{B}_{n})'(\Gamma^{f}_{n})^{-1}\bar{R}_{1}(\Gamma^{f}_{n})^{-1}
M^{B}_{n}\Big]u^{l}_{n}\no\\
&&+[B_{2}-B_{1}(\Gamma^{f}_{n})^{-1}M^{B}_{n}]'\lambda^{l}_{n}-\beta^{l}_{n}\no\\
&=&(M^{B}_{n})'(\Gamma^{f}_{n})^{-1}\bar{R}_{1}
(\Gamma^{f}_{n})^{-1}(M^{1}_{n}+B_{1}'S_{k+1}')x_{n}
+\Big[\bar{R}_{2}+(M^{B}_{n})'(\Gamma^{f}_{n})^{-1}\bar{R}_{1}(\Gamma^{f}_{n})^{-1}
M^{B}_{n}\Big]u^{l}_{n}\no\\
&&+[B_{2}-B_{1}(\Gamma^{f}_{n})^{-1}M^{B}_{n}]'\bar{P}_{n+1}
\Big\{[A-B_{1}(\Gamma^{f}_{n})^{-1}(M^{1}_{n}+B_{1}'S_{n+1}')]x_{n}+
[B_{2}-B_{1}(\Gamma^{f}_{n})^{-1}M^{B}_{n}]u^{l}_{n}\Big\}-\beta^{l}_{n}\no\\
&=&\Big\{(M^{B}_{n})'(\Gamma^{f}_{n})^{-1}\bar{R}_{1}
(\Gamma^{f}_{n})^{-1}(M^{1}_{n}+B_{1}'S_{k+1}')
+[B_{2}-B_{1}(\Gamma^{f}_{n})^{-1}M^{B}_{n}]'\bar{P}_{n+1}\no\\
&&\times
[A-B_{1}(\Gamma^{f}_{n})^{-1}(M^{1}_{n}+B_{1}'S_{n+1}')]\Big\}x_{n}
+\Big\{\Big[\bar{R}_{2}+(M^{B}_{n})'(\Gamma^{f}_{n})^{-1}\bar{R}_{1}(\Gamma^{f}_{n})^{-1}
M^{B}_{n}\Big]\no\\
&&+[B_{2}-B_{1}(\Gamma^{f}_{n})^{-1}M^{B}_{n}]'\bar{P}_{n+1}
[B_{2}-B_{1}(\Gamma^{f}_{n})^{-1}M^{B}_{n}]\Big\}u^{l}_{n}-\beta^{l}_{n}.\no
\end{eqnarray}
It yields that
\begin{eqnarray}
u^{l}_{n}&=&-(\Gamma^{l}_{n})^{-1}\mathcal{M}^{l}_{n}x_{n}
+(\Gamma^{l}_{n})^{-1}\beta^{l}_{n},\no
\end{eqnarray}
combining that linear form $u^{l}_{n}=\bar{K}_{n}x_{n}+\bar{G}_{n}x_{n-1}$, we get that
\begin{eqnarray}
(\bar{K}_{n}+(\Gamma^{l}_{n})^{-1}\mathcal{M}^{l}_{n})x_{n}+\bar{G}_{n}x_{n-1}
-(\Gamma^{l}_{n})^{-1}\beta^{l}_{n}=0.\label{121315}
\end{eqnarray}
Note that \rf{121315} holds for any $\beta^{l}_{n}$, it yields that
\begin{eqnarray}
(\bar{K}_{n}+(\Gamma^{l}_{n})^{-1}\mathcal{M}^{l}_{n})x_{n}+\bar{G}_{n}x_{n-1}=0.\no
\end{eqnarray}
That is, the optimal controller is given by
\begin{eqnarray}
u^{l}_{n}&=&\bar{K}_{n}x_{n}+\bar{G}_{n}x_{n-1}\no\\
&=&-(\Gamma^{l}_{n})^{-1}\mathcal{M}^{l}_{n}x_{n}.\label{f2.40}
\end{eqnarray}
On the other hand, consider \rf{121315} again, due to the invertibility of $\Gamma^{l}_{n}$, it yields that $\beta^{l}_{n}=0$.  Thus, by induction it is easy
to know that the controller $u^{l}_{k}$ for any $0\leq k\leq N$ has the form as \rf{f2.33}.

Next, we obtain the expression of $\lambda^{l}_{n-1}$. From \rf{f2.311}, we have that
\begin{small}\begin{eqnarray}
\lambda^{f}_{n-1}&=&\Big[\bar{Q}+(M^{1}_{n}+B_{1}'S_{k+1}')'(\Gamma^{f}_{n})^{-1}
\bar{R}_{1}(\Gamma^{f}_{n})^{-1}(M^{1}_{n}+B_{1}'S_{k+1}')\Big]x_{n}
+(M^{1}_{n}+B_{1}'S_{k+1}')'(\Gamma^{f}_{n})^{-1}
\bar{R}_{1}(\Gamma^{f}_{n})^{-1}M^{B}_{n}u^{l}_{n}\no\\
&&+[A-B_{1}(\Gamma^{f}_{n})^{-1}(M^{1}_{n}+B_{1}'S_{k+1}')]'\bar{P}_{n+1}
\Big\{[A-B_{1}(\Gamma^{f}_{n})^{-1}(M^{1}_{n}+B_{1}'S_{n+1}')]x_{n}+
[B_{2}-B_{1}(\Gamma^{f}_{n})^{-1}M^{B}_{n}]u^{l}_{n}\Big\}\no\\
&=&\Big\{\bar{Q}+(M^{1}_{n}+B_{1}'S_{k+1}')'(\Gamma^{f}_{n})^{-1}
\bar{R}_{1}(\Gamma^{f}_{n})^{-1}(M^{1}_{n}+B_{1}'S_{k+1}')
+[A-B_{1}(\Gamma^{f}_{n})^{-1}(M^{1}_{n}+B_{1}'S_{k+1}')]'\bar{P}_{n+1}\no\\
&&\times
[A-B_{1}(\Gamma^{f}_{n})^{-1}(M^{1}_{n}+B_{1}'S_{n+1}')]\Big\}x_{n}
+\Big\{(M^{1}_{n}+B_{1}'S_{k+1}')'(\Gamma^{f}_{n})^{-1}
\bar{R}_{1}(\Gamma^{f}_{n})^{-1}M^{B}_{n}\no\\
&&+[A-B_{1}(\Gamma^{f}_{n})^{-1}(M^{1}_{n}+B_{1}'S_{k+1}')]'\bar{P}_{n+1}
[B_{2}-B_{1}(\Gamma^{f}_{n})^{-1}M^{B}_{n}]\Big\}u^{l}_{n}\no\\
&=&\Big\{\bar{Q}+(M^{1}_{n}+B_{1}'S_{k+1}')'(\Gamma^{f}_{n})^{-1}
\bar{R}_{1}(\Gamma^{f}_{n})^{-1}(M^{1}_{n}+B_{1}'S_{k+1}')
+[A-B_{1}(\Gamma^{f}_{n})^{-1}(M^{1}_{n}+B_{1}'S_{k+1}')]'\bar{P}_{n+1}\no\\
&&\times
[A-B_{1}(\Gamma^{f}_{n})^{-1}(M^{1}_{n}+B_{1}'S_{n+1}')]
-(\mathcal{M}^{l}_{n})'(\Gamma^{l}_{n})^{-1}\mathcal{M}^{l}_{n}\Big\}x_{n}.\no
\end{eqnarray}\end{small}
Thus, by induction it is easy to know that the co-state $\lambda^{l}_{k}$ for any $0\leq k \leq N$ has the form as \rf{f2.34}. Now the proof of Theorem \ref{t2} is completed.
\end{proof}
\begin{remark}\label{r3}
Similar to Remark \ref{r2}, from the relevance between $x_{k}$ and $x_{k-1}$,  it cannot directly deduce that
$\bar{K}_{k}+(\Gamma^{l}_{k})^{-1}\mathcal{M}^{l}_{k}=0$ and $\bar{G}_{k}=0$. For the selection of $\bar{K}_{k}$ and $\bar{G}_{k}$, it is to be shown in the sequence.
\end{remark}

\subsection{The determination of $\bar{K}_{k}$ and $\bar{G}_{k}$}
It can be seen from Theorem \ref{t1} and \ref{t2} that the design of controller and the associated performance are based on the gain matrices  $\bar{K}_{k}$ and $\bar{G}_{k}$ which are to be determined.  In the following, we shall derive the gain matrices by a method of twice optimization for the performance of the follower, which will lead to a better performance.
\begin{theorem}\label{t3}
If $\Delta_{k}$ in \rf{2.113} is invertible, then, $\bar{K}_{k}$ and $\bar{G}_{k}$ are given as follows
\begin{eqnarray}
\bar{K}_{k}&=&-\Delta^{-1}_{k}\mathcal{M}^{2}_{k},\label{f2.44}\\
\bar{G}_{k}&=&[-(\Gamma^{l}_{k})^{-1}\mathcal{M}^{l}_{k}+\Delta^{-1}_{k}
\mathcal{M}^{2}_{k}]
[A-B_{1}(\Gamma^{f}_{k-1})^{-1}\mathcal{M}^{f}_{k-1}
-B_{2}(\Gamma^{l}_{k-1})^{-1}\mathcal{M}^{l}_{k-1}],\label{f2.55}
\end{eqnarray}
where the parameters are as in \rf{2.9}-\rf{2.11} and \rf{f2.32}-\rf{f2.332}.

In this case, the follower's optimal controller is rewritten as
\begin{eqnarray}
u^{f}_{k}&=&-(\Gamma^{f}_{k})^{-1}\mathcal{M}^{f}_{k}x_{k},\label{f2.51}
\end{eqnarray}
where
\begin{eqnarray}
\mathcal{M}^{f}_{k}&=&M^{1}_{k}+B_{1}'S_{k+1}'
-M^{B}_{k}(\Gamma^{l}_{k})^{-1}\mathcal{M}^{l}_{k}.\label{f2.53}
\end{eqnarray}

\end{theorem}
\begin{proof}
From the result of Theorem \ref{t1} and \ref{t2}, the follower's optimal strategy can be further calculated as follows.
\begin{eqnarray}
u^{f}_{k}&=&-(\Gamma^{f}_{k})^{-1}(M^{1}_{k}+M^{B}_{k}\bar{K}_{k}+B_{1}'S_{k+1}')x_{k}
-(\Gamma^{f}_{k})^{-1}M^{B}_{k}\bar{G}_{k}x_{k-1}\no\\
&=&-(\Gamma^{f}_{k})^{-1}(M^{1}_{k}+B_{1}'S_{k+1}')x_{k}
+(\Gamma^{f}_{k})^{-1}M^{B}_{k}(\Gamma^{l}_{k})^{-1}\mathcal{M}^{l}_{k}x_{k}\no\\
&=&-[\Gamma^{f}_{k})^{-1}(M^{1}_{k}+B_{1}'S_{k+1}'-M^{B}_{k}(\Gamma^{l}_{k})^{-1}
\mathcal{M}^{l}_{k}]
x_{k},\no
\end{eqnarray}
which is \rf{f2.51}.

Note that not only the leader's optimal solution is calculated based on $\bar{K}_{k}$ and $\bar{G}_{k}$, but also the follower's optimal solution. Hence, it is necessary to show the calculation of $\bar{K}_{k}$ and $\bar{G}_{k}$. For selecting the better expression, consider the follower's cost value \rf{f2.21} which is related to $\bar{K}_{k}$, we select the parameters $\bar{K}_{k}$ and $\bar{G}_{k}$ to minimize \rf{f2.21} more. Hence,
\begin{eqnarray}
0=\frac{\partial P_{k}}{\partial \bar{K}_{k}}&=&R_{2}\bar{K}_{k}+B_{2}'P_{k+1}(A+B_{2}\bar{K}_{k})
+B_{2}'S_{k+1}'-(M^{B}_{k})'(\Gamma^{f}_{k})^{-1}
(M^{1}_{k}+M^{B}_{k}\bar{K}_{k}+B_{1}'S_{k+1}')\no\\
&=&[R_{2}+B_{2}'P_{k+1}B_{2}-(M^{B}_{k})'(\Gamma^{f}_{k})^{-1}M^{B}_{k}]\bar{K}_{k}
+[B_{2}'P_{k+1}A+B_{2}'S_{k+1}'\no\\
&&-(M^{B}_{k})'(\Gamma^{f}_{k})^{-1}
(M^{1}_{k}+B_{1}'S_{k+1}')]\no\\
&=&\Delta_{k}\bar{K}_{k}+\mathcal{M}^{2}_{k},\label{f2.42}
\end{eqnarray}
From \rf{f2.42}, when $\Delta_{k}$ is invertible, obviously, it yields
\begin{eqnarray}
\bar{K}_{k}=-\Delta^{-1}_{k}\mathcal{M}^{2}_{k}.\label{f2.44}
\end{eqnarray}
Furthermore, consider the given linear constraint \rf{f2.6} and the optimal result \rf{f2.33}, we get
\begin{eqnarray}
\bar{G}_{k}x_{k-1}&\equiv&u^{l}_{k} -\bar{K}_{k}x_{k}\no\\ &=&-(\Gamma^{l}_{k})^{-1}\mathcal{M}^{l}_{k}x_{k}+\Delta^{-1}_{k}\mathcal{M}^{2}_{k}x_{k}\no\\
&=&[ -(\Gamma^{l}_{k})^{-1}\mathcal{M}^{l}_{k}+\Delta^{-1}_{k}\mathcal{M}^{2}_{k}]
[A-B_{1}(\Gamma^{f}_{k-1})^{-1}\mathcal{M}^{f}_{k-1}
-B_{2}(\Gamma^{l}_{k-1})^{-1}\mathcal{M}^{l}_{k-1}]x_{k-1}.\no
\end{eqnarray}
Due to the fact that the above relationship holds for any $x_{k-1}$, therefore, the selection of  $\bar{G}_{k}$ is as in \rf{f2.55}.
\end{proof}

\section{Simulation}
We consider the above Stackelberg problem with the following coefficients:
$A=3, B_1=1, B_2=1, Q=2, R_1=1, R_2=10, \bar{Q}=1, \bar{R}_1=1, \bar{R}_2=2$. Given the  terminal values $P(N+1)=\bar{P}(N+1)=2$ and an initial value $x_0=5$. Next, we will give the results of feedback case and liner closed-loop case.

First, for the standard feedback case, i.e., $u^{f}_{k}=K_{k}x_{k}, u^{l}_{k}=\bar{K}_{k}x_{k}$, by calculating, the
optimal strategies are as follows (Table 1).
\begin{center}
	\textbf{Table 1}~~~the feedback results\\[5pt]
	\renewcommand\arraystretch{2}
	\setlength{\tabcolsep}{1.2mm}{
		        \begin{tabular}{|c|c|c|c|c|c|c|}
            \hline
           Step k &0&1&2&3&4&5\\
            \hline
            $u^{f}_{k}$ & $ -1.9350x_{0}$  &$-1.9349x_{1}$ & $-1.9348x_{2}$ &$-1.9333x_{3}$&$-1.9107x_{4}$&$-1.5000x_{5}$\\
  \hline
$u^{l}_{k}$& $-0.9335 x_{0}$ &$-0.9335x_{1}$ & $-0.9334x_{2}$ &$-0.9322x_{3}$&$-0.9156
x_{4}$&$-0.7500x_{5}$\\
            \hline
        \end{tabular}}
\end{center}
And in this case, the optimal cost values are given by $J^{f\ast}=367.8335, J^{l\ast}=165.0297$.

On the other hand, for the linear closed-loop case, i.e., $u^{f}_{k}=K_{k}x_{k}+G_{k}x_{k-1},\
u^{l}_{k}=\bar{K}_{k}x_{k}+\bar{G}_{k}x_{k-1}$, by deriving the gain matrices, it can be written as a form of a current state feedback, and the results are calculated as follows (Table 2).
\begin{center}
\textbf{Table 2}~~~the linear closed-loop results \\[7pt]
\renewcommand\arraystretch{2}
\setlength{\tabcolsep}{1.2mm}{
        \begin{tabular}{|c|c|c|c|c|c|c|}
            \hline
            Step k &0&1&2&3&4&5\\
            \hline
            $u^{f}_{k}$ & $-1.8124x_{0}$  &$-1.8136x_{1}$ & $-1.8889x_{2}$ &$-1.8873x_{3}$&$-1.8576x_{4}$&$-1.5000x_{5}$\\
            \hline
$u^{l}_{k}$& $ -0.9062 x_{0}$ &$-0.9062x_{1}$ & $ -0.9068x_{2}$ &$ -0.9142x_{3}$&$-0.8988
x_{4}$&$-0.7500x_{5}$\\
            \hline
        \end{tabular}}
\end{center}
where the results of $\bar{K}_{k}$ and $\bar{G}_{k}$ are given:
\begin{center}
	\textbf{Table 3}~~~values of $\bar{K}_{k}$ and $\bar{G}_{k}$ \\[5pt]
	\renewcommand\arraystretch{2}
	\setlength{\tabcolsep}{1.2mm}{
		        \begin{tabular}{|c|c|c|c|c|c|c|}
            \hline
           Step k &0&1&2&3&4&5\\
            \hline
            $\bar{K}_{k}$ & $ -0.9062$  &$-0.2497$ & $-0.2493$ &$-0.2489$&$-0.2437
$&$-0.1875$\\
  \hline
$\bar{G}_{k}$&  &$-0.1876$ & $ -0.1874$ &$ -0.1860$&$-0.1300$&$ -0.1370$\\
            \hline
        \end{tabular}}
\end{center}
It should be pointed out that the gain matrices $\bar{K}_{k}, \bar{G}_{k}$ given in Table 3 satisfy the relationship $\bar{K}_{k}x_{k}+\bar{G}_{k}x_{k-1}=K_{k}x_{k}$, where $K_{k}$ is the corresponding gain matrix in the form of state feedback in Table 2. And the optimal cost values of linear closed-loop case are $J^{f\ast}=366.4332,J^{l\ast}=160.9312$.

Comparing the cost values of the feedback case and the linear closed-loop case, it is not hard to note that the feedback result may not always minimize the cost values of Stackelberg problem. Actually, it is rational due to that more information is obtained by players for the linear closed-loop case.

\section{Conclusion}

In this paper, we studied the the closed-loop Stackelberg strategy for linear-quadratic leader-follower game. The linear closed-loop Stackelberg strategy with one-step memory has been explicitly given in terms of Riccat equations. The key technique is to apply the constrained maximum principle to the leader-follower game and explicitly solve the corresponding forward and backward difference equations. Numerical examples show that the derived strategy achieves better performance than feedback strategy.

\bibliographystyle{plain}

%
\end{document}